\documentclass[a4paper]{article}

\usepackage[english]{babel}
\usepackage[utf8]{inputenc}
\usepackage{amsmath}
\usepackage{graphicx}
\usepackage{amsfonts}
\usepackage{enumitem}
\usepackage{amssymb}
\usepackage[colorinlistoftodos]{todonotes}
\usepackage{geometry}
\usepackage{amsthm}
\usepackage{verbatim}
\usepackage{hyperref}
\graphicspath{{graphics/}}

\theoremstyle{plain}
\newtheorem{thm}{Theorem}

\newtheorem{conj}[thm]{Conjecture}

\theoremstyle{definition}

\title{A generalisation of Seymour's second neighbourhood conjecture}
\author{Jeck Lim\footnote{Trinity College, Cambridge CB2 1TQ, UK.\quad Email: jl945@cam.ac.uk}}
\date{January 19, 2020}

\begin{document}

\maketitle

\begin{abstract}
In this note we propose a generalisation of Seymour's Second Neighbourhood Conjecture to two directed graphs on a vertex set. We prove that this generalisation holds in the case of tournaments, and we show that a natural strengthening of this conjecture does not hold.
\end{abstract}

%\begin{figure}[ht]
%\includegraphics[width=10em]{name}
%\centering
%\end{figure}

\section{Introduction}

In this note, all graphs are finite. \textit{Directed graphs} (or \textit{digraphs}) do not contain parallel edges, but may contain self-loops. \textit{Oriented graphs} are directed graphs with no self-loops and directed 2-cycles. A digraph $G$ on a vertex set $V=V(G)$ can be represented by its set of edges, $E(G)$, a subset of $V\times V$. The \textit{first neighbourhood} or \textit{out-neighbourhood} of a vertex $v$ in $G$ is the set $\{u\in V(G)\mid (v,u)\in E(G)\}$, denoted by $N_G^+(v)$. The \textit{in-neighbourhood} of $v$ is the set $\{u\in V(G)\mid (u,v)\in E(G)\}$, denoted by $N_G^-(v)$. We set $d_G^+(v)=|N_G^+(v)|$ and $d_G^-(v)=|N_G^-(v)|$. The \textit{second neighbourhood} or \textit{second out-neighbourhood} of a vertex $v\in V(G)$ is the set $(\bigcup_{u\in N_G^+(v)}N_G^+(u))\backslash N_G^+(v)$, denoted by $N_G^{++}(v)$. We also set $d_G^{++}(v)=|N_G^{++}(v)|$. Similarly define the \textit{second in-neighbourhood} $N_G^{--}(v)$ and $d_G^{--}(v)$. We omit the subscripts $G$ if the context is clear. In 1990, Seymour conjectured the following statement:
\begin{conj}[Seymour, see \cite{dean}]
Every oriented graph has a vertex $v$ satisfying $d^{++}(v)\geq d^+(v)$.
\end{conj}
A vertex $v$ of $G$ is said to satisfy the \textit{second neighbourhood property (SNP)} if $d^{++}(v)\geq d^+(v)$. $G$ is said to satisfy the \textit{second neighbourhood conjecture (SNC)} if it has a vertex satisfying SNP.

In 1996, Fisher \cite{tournament} proved the conjecture for tournaments, i.e. oriented graphs with an edge between every pair of vertices.

It will be convenient to consider also a weighted version of SNC. Suppose $G$ is weighted by a non-negative real-valued function $\omega:V(G)\to \mathbb{R}_{\geq 0}$. The \textit{weight} of a set of vertices is the sum of the weights of its members. We say that a vertex $v$ of $G$ satisfies the \textit{weighted second neighbourhood property (WSNP)} if $\omega(N^{++}(v))\geq \omega(N^+(v))$. We say $G$ satisfies the \textit{weighted second neighbourhood conjecture (WSNC)} if for every such function $\omega$, there is a vertex $v$ satisfying WSNP.

\begin{thm}
The following are equivalent:
\begin{enumerate}
\item Every oriented graph satisfies the second neighbourhood conjecture.
\item Every oriented graph satisfies the weighted second neighbourhood conjecture.
\end{enumerate}
\end{thm}
\begin{proof}
Clearly (2) implies (1). Supposing (1), we first show WSNC holds for positive integer weights $\omega:V(G)\to \mathbb{N}$, hence holds for positive rational weights. By continuity, WSNC holds for all non-negative real weights.

Given a weight $\omega:V(G)\to \mathbb{N}$, consider the graph $G'$ formed by duplicating each vertex $v$ $\omega(v)$-many times to get vertices $v_1,\ldots,v_{\omega(v)}$, with edges $(u_i,v_j)$ whenever $(u,v)$ is an edge of $G$, over all possible $i,j$. We call this process \textit{blowing up} vertex $v$ with weight $\omega(v)$. Then $G$ having the WSNP is equivalent to $G'$ having the SNP, hence $G$ satisfies the WSNC.
\end{proof}

\section{A generalisation}
We start with a generalisation which turns out to be false and give a counterexample with 36 vertices. Then we give a modification of the generalisation which we believe is true.

Let $A,B$ be digraphs on the same vertex set $V$. Recall that a digraph can be viewed as a subset of $V\times V$, so standard set operations can be performed on them. Let $I=\{(v,v)\mid v\in V\}$ be the \textit{identity graph}. The \textit{transpose} (or \textit{inverse}) of $A$ is defined by $A^T=\{(u,v)\mid (v,u)\in A\}$, i.e. the graph with all edges of $A$ reversed. We define the \textit{product graph} $AB$ to be the subset $\{(u,v)\mid \exists w\in V,(u,w)\in A,(w,v)\in B\}$. Set $A(v)=\{u\in V\mid (v,u)\in A\}$ for $v\in V$. 

The second neighbourhood conjecture can be reformulated in the following way: $A$ is a directed graph with $A\cap A^T=I$. Then there is a vertex $v$ such that $|AA(v)|\geq 2|A(v)|-1$. 

This leads to a natural generalisation: Let $A,B$ be directed graphs on a vertex set $V$ such that $A\cap B^T=I$. Is there a vertex $v$ such that $|AB(v)|\geq |A(v)|+|B(v)|-1$? The original SNC is just with $A=B$.

Similar to SNC, this has a weighted version $\omega(AB(v))\geq \omega(A(v))+\omega(B(v))-\omega(v)$, which is equivalent to the unweighted one. However, the weighted version turns out to be false. We provide 2 counterexamples, both with $A,B$ having no 2-cycles.

Our first counterexample has $A\subset B$. We give the weights of each vertex and out-neighbours of $A,B,AB$ in the table below:
\begin{center}
\begin{tabular}{c|c|c|c|c}
V & weight & $A$ & $B$ & $AB$\\\hline
1 & 7 & 1,2,5,6 & 1,2,5,6 & 1,2,3,4,5,6\\
2 & 3 & 2,3 & 2,3,4 & 1,2,3,4,5\\
3 & 11 & 1,3,4,5 & 1,3,4,5 & 1,2,3,4,5,6\\
4 & 3 & 1,4 & 1,4,6 & 1,2,4,5,6\\
5 & 3 & 2,5,6 & 2,4,5,6 & 2,3,4,5,6\\
6 & 9 & 2,3,6 & 2,3,6 & 1,2,3,4,5,6
\end{tabular}
\end{center}
Thus by blowing up the vertices with the appropriate weights, we obtain a counterexample to the unweighted generalisation with 36 vertices. Our second counterexample has $B\subset A$, described below:
\begin{center}
\begin{tabular}{c|c|c|c|c}
V & weight & $A$ & $B$ & $AB$\\\hline
1 & 17 & 1,2,5,6 & 1,2,5,6 & 1,2,3,4,5,6\\
2 & 11 & 2,3,4 & 2,3,4 & 1,2,3,4,5\\
3 & 15 & 1,3,6 & 1,3 & 1,2,3,5,6\\
4 & 8 & 1,3,4,5 & 3,4,5 & 1,2,3,4,5,6\\
5 & 5 & 2,3,5 & 2,3,5 & 1,2,3,4,5\\
6 & 8 & 2,4,5,6 & 2,5,6 & 2,3,4,5,6
\end{tabular}
\end{center}
This yields an unweighted counterexample with 64 vertices after blowing up.

By replacing $AB$ with $AB\cup BA$, we obtain a weaker generalised conjecture:
\begin{conj}
Let $A,B$ be directed graphs on a vertex set $V$ such that $A\cap B^T=I$. Then there is a vertex $v$ such that $|(AB\cup BA)(v)|\geq |A(v)|+|B(v)|-1$.
\end{conj}
The weighted version is as follows:
\begin{conj}
Let $A,B$ be directed graphs on a vertex set $V$ such that $A\cap B^T=I$. Then for any non-negative weight function $\omega$, there is a vertex $v$ such that $\omega((AB\cup BA)(v))\geq \omega(A(v))+\omega(B(v))-\omega(v)$.
\end{conj}
This generalises SNC by setting $A=B$. By replacing $A,B$ with $A\cap B,A\cup B$, we may assume that $A\subset B$. The above counterexamples show that the union $AB\cup BA$ is required, even when $A\subset B$ or $B\subset A$. The tournament version of the above conjecture holds; we give a proof based on the technique of winning and losing densities of Fisher \cite{tournament}:
\begin{thm}
Let $A,B$ be directed graphs on a vertex set $V$ such that $A\cap B^T=I$ and $A\cup B^T=V\times V$. Let $\omega:V\to\mathbb{R}_{\geq 0}$ be a weight function. Then there is a vertex $v$ such that $\omega((AB\cup BA)(v))\geq \omega(A(v))+\omega(B(v))-\omega(v)$.
\end{thm}
\begin{proof}
We first assume that $A\subset B$. Set $C=AB\cup BA$. Consider the oriented graph $G$ with edges $\{(u,v)\mid u\neq v,(u,v)\in A\}$. By Theorem 1 of \cite{tournament}, $G$ has a losing density. A losing density $l$ is a weight function satisfying $l(N_G^+(v))\geq l(N_G^-(v))$ for all vertex $v$. Further, if $l(v)>0$, then $l(N_G^+(v))= l(N_G^-(v))$. Fix any vertex $v$, set $S_1=N_G^-(v)=A^T(v)\setminus\{v\},S_2=C^T(v)\setminus B^T(v)$. We show that $l(S_2)\geq l(S_1)$. Let $Q$ be the subgraph $V\setminus C^T(v)\cup\{v\}$. This partitions $V$ into the sets $S_1\cup S_2\cup Q\cup (B^T(v)\setminus A^T(v))$. If $l(Q)=0$, then $S_1=N_G^-(v)$ and $N_G^+(v)\subset S_2\cup Q$, thus $l(S_1)\leq l(S_2\cup Q)=l(S_2)$. Otherwise, we have $l(Q)>0$. Within $Q$, we have
\begin{align*}
\sum_{v\in V(Q)}l(v)l(N_Q^-(v)) &= \sum_{v\in V(Q)} \sum_{u\in N_Q^-(v)} l(v)l(u)\\
&=\sum_{v\in V(Q)}\sum_{u\in N_Q^+(v)} l(v)l(u) =\sum_{v\in V(Q)}l(v)l(N_Q^+(v)).
\end{align*}
Thus we have $l(N_Q^-(u))\geq l(N_Q^+(u))$ for some $u\in Q$ with $l(u)>0$. For each $w\in S_1$, we cannot have $(w,u)\in B^T$, since otherwise we have $u\in A^TB^T(v)$. Since $A\cup B^T=V\times V$, we must have $(u,w)\in A^T$, hence $N_G^-(u)\supset N_Q^-(u)\cup S_1$. For each $x\in B^T(v)\setminus A^T(v)$, we cannot have $(x,u)\in A^T$, since otherwise $u\in B^TA^T(v)$. Thus we get $N_G^+(u)\subset N_Q^+(u)\cup S_2$. Since $l(u)>0$, $l(N_G^+(u))=l(N_G^-(u))$. From $l(N_Q^-(u))\geq l(N_Q^+(u))$, we deduce that $l(S_2)\geq l(S_1)$.

We want to show there is a $v$ such that $\omega(C(v)\setminus B(v))\geq \omega(A(v)\setminus\{v\})$. In fact, we show stronger statement that
$$\sum_{v\in V} l(v)(\omega(C(v)\setminus B(v))-\omega(A(v)\setminus\{v\}))\geq 0,$$
then we are done. The above sum is equal to
$$\sum_{v\in V} \omega(v)(l(C^T(v)\setminus B^T(v))-l(A^T(v)\setminus\{v\})).$$
But the sets $C^T(v)\setminus B^T(v)$ and $A^T(v)\setminus\{v\}$ are just $S_2$ and $S_1$ defined above corresponding to $v$, and $l(S_2)\geq l(S_1)$, thus the above sum is non-negative.
\end{proof}
We remark that, setting $A=B$, we recover the usual proof by Fisher \cite{tournament} of the SNC for tournaments.


\begin{thebibliography}{9}
\bibitem{dean}
N. Dean and B. J. Latka.
\textit{Squaring the tournament – an open problem}.
Congressus Numerantium \textbf{109} (1995), 73-80.

\bibitem{tournament}
D. Fisher.
\textit{Squaring a Tournament: A Proof of Dean's Conjecture}.
J. Graph Theory \textbf{23} (1996), 43-48.

\bibitem{pointsymmetric}
Y. O. Hamidoune.
\textit{On iterated image size for point-symmetric relations}.
https://arxiv.org/abs/0704.0459 (2007).

\bibitem{olson}
J. E. Olson.
\textit{On the Sum of Two Sets in a Group}.
Journal of Number Theory \textbf{18} (1984), 110-120.

\bibitem{kemperman}
J. H. B. Kemperman.
\textit{On complexes in a semigroup}.
Indag. Math. \textbf{18} (1956), 247-254.

\begin{comment}
\bibitem{miss_star}
S. Ghazal.
\textit{Seymour’s second neighborhood conjecture for tournaments missing a generalized star}.
https://arxiv.org/abs/1106.0085
\end{comment}

\end{thebibliography}
\end{document}